\documentclass[12pt,a4paper]{amsart}    
\usepackage{graphicx}
\usepackage{a4wide}
\usepackage{mathtools}
\usepackage{amssymb}
\usepackage{amsmath} 
\usepackage{amscd} 
\usepackage[all]{xypic}
\usepackage{url}
\usepackage{color}
\usepackage{multicol}

\newcommand{\fh}{\mathfrak h}

\newcommand{\R}{\mathbb{R}}
\newcommand{\C}{\mathbb{C}}
\newcommand{\Z}{\mathbb{Z}}
\newcommand{\NN}{\mathbb{N}}

\newcommand{\qa}{\mathbb{H}} 

\newcommand{\SO}{\frak{so}}

\newtheorem{prop*}{Proposition}

\newtheorem{thm*}{Theorem}
\newtheorem{lemma*}{Lemma}

\newtheorem{cor*}{Corollary}

\newtheorem{rem*}{Remark}
\newtheorem{def*}{Definition}

\makeatletter
\@namedef{subjclassname@2020}{%
  \textup{2020} Mathematics Subject Classification}
\makeatother

\begin{document}

\title[]{
Local control on quaternionic Heisenberg group of dimension $7$
}


\subjclass[]{}

 \author{Jan Eisner and 
Lenka Zalabov\'a
}

\address{
J.E.:Institute of Mathematics, 
Faculty of Science, University of South Bohemia, 
	Bra\-ni\-\v sov\-sk\' a 1760, 370 05 \v Cesk\' e Bud\v ejovice, 
Czechia\\
L.Z.:Institute of Mathematics, 
Faculty of Science, University of South Bohemia, 
	Bra\-ni\-\v sov\-sk\' a 1760, 370 05 \v Cesk\' e Bud\v ejovice, 
 and
     Department of Mathematics and Statistics, 
 	Faculty of Science, Masaryk University,
 	Kotl\' a\v rsk\' a 2, 611 37 Brno, 
Czech Republic
}

\email{jeisner@prf.jcu.cz, lzalabova@gmail.com
}	

\thanks{The authors thank to Jan Gregorovič and Petr Nečesal for useful discussions during the work on the paper. L.Z is supported by the grant GACR 24-10887S Cartan supergeometries and Higher Cartan geometries.}

\maketitle

\thispagestyle{empty}

\vspace{7pt}

\begin{abstract}
We describe the quaternionic Heisenberg group in the dimension $7$ explicitly as a matrix group. We study the local control of a compatible left-invariant control system.
We describe the impact of symmetries of the corresponding sub-Riemannian structure on the optimality of geodesics.
\end{abstract}

\section{Introduction}

The three-dimensional Heisenberg group, a $2$-step nilpotent Lie group of the filtration $(2,3)$, is the most known and studied model in the control theory and the sub-Riemannian geometry, \cite{ABB,donne}. In fact, it can be viewed as a flat, aka maximally symmetric, model in the sub-Riemannian geometry, \cite{zel}. In applications, it can be also viewed as a nilpotent approximation of the problem of the vertical rolling disk, \cite{b96,hzn}. 

There are various generalizations to study. Clearly, it is natural to consider Heisenberg groups of higher dimensions,  i.e. $2$-step nilpotent groups of filtrations $(2n,2n+1)$, or the products of $n$ copies of the three--dimensional Heisenberg groups, \cite{mope}. The other possibility is to focus on the free $2$-step distributions, i.e. filtrations $(n,\frac{n}{2}(n+1))$; the most known and studied is the case $(3,6)$, \cite{Rizzi, mya1, momo, mope36}. 
We are interested in another generalization; we focus here on the Heisenberg group of the lowest possible dimension over quaternions. This leads to a specific $2$-step nilpotent group of the filtration $(4,7)$.  

It follows from \cite{Biquard, parabook} that quaternionic Heisenberg groups are flat models of quaternionic contact structures which are Cartan geometries modeled over $\frak{sp}(p,q)$. We use this fact to describe the quaternionic Heisenberg algebra of the filtration $(4,7)$ and the corresponding Carnot group $\mathcal{H}$ via matrices from $\frak{sp}(2,1)$ in Section \ref{sec1.1}. We use the Maurer-Cartan form on $\mathcal{H}$ to define the non-holonomic distribution and the sub-Riemannian structure  compatible with the group structure in Section \ref{sec1.2}. We study symmetries of the sub-Riemannian structure both algebraically and infinitesimally in Section \ref{sec1.3}.

We formulate and solve the control problem corresponding to the sub-Riemannian structure on $\mathcal{H}$ in Section \ref{sec2.1}. We employ Hamiltonian viewpoint to approach the control problem and use its invariancy with respect to the group action, \cite{ABB}. 
In particular, we describe the Hamiltonian in suitable coordinates and find corresponding control functions in Section~\ref{sec2.1} and we describe geodesics starting at the origin in Section~\ref{sec2.2}.

Finally, we give our main results in  Sections \ref{sec3.1} and \ref{sec3.2}. In particular, we study the action of symmetries on geodesics and we use the action to discuss their optimality. 
It turns out that the behavour of geodesics is pretty analogous to classical Heisenberg, \cite{mope}. In Section \ref{sec3.3} we then give some final comments and visualizations.

\section{Quaternionic Heisenberg group}
\label{sec1}
Let us consider the quaternion algebra $\mathbb{H}$ and write
$q=q_1+q_2I+q_3J+q_4K$
 for a quaternion $q \in \qa$. We simultaneously use the matrix representation of the quaternion $q$ as
\begin{align}
\label{quaternion}
q= \left[ \begin {smallmatrix} q_{{1}}&-q_{{2}}&-q_{{3}}&-q_{{4}}
\\ q_{{2}}&q_{{1}}&-q_{{4}}&q_{{3}}
\\ q_{{3}}&q_{{4}}&q_{{1}}&-q_{{2}}
\\ q_{{4}}&-q_{{3}}&q_{{2}}&q_{{1}}\end {smallmatrix}
 \right].
\end{align}

\subsection{Description of Lie algebra and Lie group}
\label{sec1.1}
The quaternionic Heisenberg group is a maximally symmetric model of quaternionic structures, \cite[Section 4.3.3]{parabook}.  
These are Cartan geometries modeled on $\frak{sp}(p+1,q+1)$ and we focus on (real) dimension $7$, i.e. we start with $\frak{sp}(2,1)$. We consider the matrix representation of the Lie algebra $\frak{sp}(2,1)$ as follows
\begin{align}
\label{sp}
\left[
\begin{smallmatrix}
c & h & \Im(f) \\
b & \Im(d) & -\bar h \\
\Im(a) & -\bar b & -\bar c
\end{smallmatrix}
\right]
\end{align}
for $a,b,c,d,e,f,h \in \qa$, where 
$$\bar q=q_1-q_2I-q_3J-q_4K, \ \ \ \Im(q)=q_2I+q_3J+q_4K$$ denote the conjugated quaternion of $q$ and its imaginary part. Having this convention, the algebra $\frak{sp}(2,1)$ carries a natural $|2|$-grading whose components are according to the diagonal and the minus part of the grading being below the main diagonal.  

We model the quaternionic Heisenberg algebra $\fh$ as the minus part of the grading as follows
\begin{align*}
(b,\Im(a))=\left[
\begin{smallmatrix}
0 & 0 & 0 \\
b & 0 & 0 \\
\Im(a) & -\bar b & 0
\end{smallmatrix}
\right] \subset \frak{sp}(2,1)
\end{align*}
for (real) coordinates $a_2,a_3,a_4,b_1,b_2,b_3,b_4$. In particular, the Lie bracket equals to $$[(b,\Im(a)),(d,\Im(c))]=(0,-2\Im(\bar b d)).$$
Viewing the matrix interpretation as exponential coordinates around the origin, we can describe the corresponding group structure of the  quaternionic Heisenberg group $\mathcal{H}$ with the Lie algebra $\fh$. Indeed, direct matrix computation gives 
\begin{align}
\begin{split}
\label{operace}
(b,\Im(a))\cdot (d,\Im(c))&=\log(\exp(b,\Im(a))\exp(d,\Im(c)))\\ &=(b+d,\Im(a)+\Im(c)-\Im(\bar b d)).
\end{split}
\end{align}
From now on, we employ the following definition.
\begin{def*}
A \emph{quaternionic Heisenberg group} is a Lie group $\mathcal{H}$ of elements of the form 
$(b,\Im(a))$ together with the group multiplication \eqref{operace}.
\end{def*}

\subsection{Maurer-Cartan form and non-holonomic distribution} \label{sec1.2}
Let us describe here the Maurer-Cartan form $\omega_\mathcal{H}$ on $\mathcal{H}$ to give a more geometric interpretation of the structure. Using the matrix description of elements of $\mathcal{H}$ and the group operation \eqref{operace}, we find $\omega_{\mathcal{H}}=h^{-1}\cdot dh$ as 

$$
\omega_{\mathcal{H}}=
\left[
\begin{smallmatrix}
0 & 0 & 0 \\
\beta & 0 & 0 \\
\alpha & -\bar \beta & 0 
\end{smallmatrix}
\right], 
\ \ \ 
\beta=\left[
\begin{smallmatrix}
 db_1 &*&*& * \\
 db_2 &*&*& * \\
 db_3 &*&*& *  \\
 db_4 &*&*& * 
 \end{smallmatrix}
\right],
\ \ \ 
\alpha=\left[
\begin{smallmatrix}
 0 &*&*& * \\
 da_2-b_2 db_1+b_1 db_2+b_4db_3-b_3db_4 &*&*& * \\
 da_3 - b_3db_1-b_4 db_2+b_1 db_3+b_2 db_4 &*&*& *  \\
 da_4 - b_4 db_1+b_3 db_2- b_2 db_3+b_1 db_4&*&*& * 
 \end{smallmatrix}
\right],
$$
where positions $*$ are determined by the matrix representation of the  quaternion \eqref{quaternion}.
In particular, the corresponding coframe allows us to determine a $4$-dimensional distribution on $\mathcal{H}$ as follows.

\begin{prop*}
The natural $4$-dimensional left-invariant distribution $H$ on the
quaternionic Heisenberg group $\mathcal{H}$ is generated by the vector fields
\begin{align*}
e_1& =\partial_{b_1}+b_2\partial_{a_2}+b_3\partial_{a_3}+b_4\partial_{a_4} \\
e_2 & = \partial_{b_2}-b_1\partial_{a_2}+b_4\partial_{a_3}-b_3\partial_{a_4} \\
e_3 & =\partial_{b_3}-b_4\partial_{a_2}-b_1\partial_{a_3}+b_2\partial_{a_4} \\
e_4 & = \partial_{b_4}+b_3\partial_{a_2}-b_2\partial_{a_3}-b_1\partial_{a_4}
\end{align*}
where the notation $\partial_x$ stands for $\frac{\partial}{\partial_x}$.
\end{prop*}
\begin{proof} 
Viewing the first columns of $\alpha, \beta$ as a (real) coframe,
one finds the vector fields $e_i$ as the $\beta$-part of the corresponding dual frame (while the $\alpha$-part consists just of $\partial_{a_i}$ for $i=2,3,4$). It follows from the properties of the Maurer-Cartan form that both the frame and the coframe are left-invariant with respect to the group multiplication.
\end{proof}
The distribution $H$ is non-holonomic. Denoting 
$$e_5=\partial_{a_2},\ e_6=\partial_{a_3},\ e_7=\partial_{a_4},$$ 
we get the following Lie algebra structure that clearly reflects the structure of $\fh$ and forms a $2$-step nilpotent Lie algebra
\begin{align}
\begin{split}
\label{zavorky}
[e_1,e_2]&=-2e_5, \ [e_1, e_3]=-2e_6, \ [e_1,e_4]=-2e_7, \\
[e_2, e_3]&=\ \ 2e_7, \ [e_2,e_4]=-2e_6, \ [e_3, e_4]=\ \ 2e_5.
\end{split}
\end{align}  
\subsection{Sub-Riemannian structure and its symmetries} \label{sec1.3}
Let us consider the left-invariant sub-Riemannian metric on $H$ by declaring $e_i$, $i=1,2,3,4$, mutually orthonormal, i.e.
$$
g=db_1^2+db_2^2+db_3^2+db_4^2.
$$ 
Altogehter, we get a left-invariant sub-Riemannian structure $(\mathcal{H},H,g)$ on the quaternionic Heisenberg group of dimension $7$. Before we focus on the corresponding control problem, we look at symmetries of the sub-Riemannian structure.

Let us firstly employ the algebraic viewpoint.
Symmetries of the quaternionic contact structure $(\mathcal{H},H)$ on the quaternionic Heisenberg group $\mathcal{H}$  are (left multiplications by) elements of $Sp(2,1)$. On the Lie algebra level, we focus on elements of $\frak{sp}(2,1)$. In particular, all translations along  $\mathcal{H}$ correspond to the minus part of the grading described in \eqref{sp}. In addition to the translations, isotropy symmetries of the sub-Riemannian metric $g$ may appear. If so, they generally correspond to a subalgebra of $\frak{so}(4,\R)$, \cite{ams}. In particular, isotropy symmetries of $g$ must be contained in the zero part of the grading described in \eqref{sp} and the positive part does not occur. The zero part takes form 
$$\mathbb{H}\oplus \frak{sp}(1)=\R\oplus \frak{sp}(1)\oplus\frak{sp}(1),$$ 
and the real part corresponds to the grading element whose action clearly does not preserve the metric. We end up with
\begin{align}
\label{mat-sym}
\left[
\begin{smallmatrix}
\Im(c) & 0 & 0 \\
b & \Im(d) & 0 \\
\Im(a) & -\bar b & \Im(c)
\end{smallmatrix}
\right] \subset \frak{sp}(2,1).
\end{align}

Let us remind that an infinitesimal symmetry of $(\mathcal{H},H,g)$ is a vector fields $v$ such that $\mathcal{L}_vg=0$ and $\mathcal{L}_vH \subset H$ for the Lie derivative $\mathcal{L}$. Flows of all such vector fields form a connected component of the symmetry group of $(\mathcal{H},H,g)$.
We show that the Lie algebra of infinitesimal symmetries corresponds to the algebra \eqref{mat-sym} and find explicitly the corresponding vector fields. 

\begin{prop*}
\label{prop-symmetries}
The Lie algebra of infinitesimal symetries of the left-invariant sub-Rie\-mann\-ian structure $(\mathcal{H},H,g)$
is generated by the following vector fields.
Translations are generated by the vector fields  
\begin{align*}
t_1&= \partial_{a_2} \\
t_2&=\partial_{a_3}\\
t_3&=\partial_{a_4}\\
t_4& =\partial_{b_1}-b_2\partial_{a_2}-b_3\partial_{a_3}-b_4\partial_{a_4}\\
t_5& =\partial_{b_2}+b_1\partial_{a_2}-b_4\partial_{a_3}+b_3\partial_{b_4}\\
t_6& =\partial_{b_3}+b_4\partial_{a_2}+b_1\partial_{a_3}-b_2\partial_{a_4} \\
t_7& =\partial_{b_4}-b_3\partial_{a_2}+b_2\partial_{a_3}+b_1\partial_{a_4}.
\end{align*}
In particular, these are right-ivariant vector fields on $\mathcal{H}$, i.e. they commute with $e_i$, $i=1, \dots, 7$.
Moreover, there is an algebra of infinitesimal symmetries stabilizing the origin $o=(0,\dots,0) \in \mathcal{H} $ generated by
\begin{align*}
s_1&=\ \ 2a_4 \partial_{a_3}-2a_3 \partial_{a_4}-b_2 \partial_{b_1}+b_1 \partial_{b_2}+b_4 \partial_{b_3}-b_3 \partial_{b_4}\\
s_2&=-2a_4 \partial_{a_2}+2a_2\partial_{a_4}-b_3 \partial_{b_1}-b_4 \partial_{b_2}+b_1 \partial_{b_3}+b_2 \partial_{b_4}\\
s_3&=\ \ 2a_3\partial_{a_2}-2a_2 \partial_{a_3}-b_4\partial_{b_1}+b_3\partial_{b_2}-b_2 \partial_{b_3}+b_1 \partial_{b_4}\\
s_4&= \ \ b_2\partial_{b_1}-b_1\partial_{b_2}+b_4\partial_{b_3}-b_3\partial_{b_4}\\
s_5 &= \ \ b_3\partial_{b_1}-b_4\partial_{b_2}-b_1\partial_{b_3}+b_2\partial_{b_4}\\
s_6 &=\ \ b_4\partial_{b_1}+b_3\partial_{b_2}-b_2\partial_{b_3}-b_1\partial_{b_4}
\end{align*}
that forms $\frak{sp}(1) \oplus \frak{sp}(1) \simeq \frak{so}(4,\R)$.
\end{prop*}
\begin{proof}
Let us roughly summarize the computation, all the details can be treated by hand or using some CAS system; we used Maple.
We consider a general vector field $v=\sum_{i=1}^7v_ie_i$ for arbitrary functions $v_i=v_i(a_2,a_3,a_4,b_1,b_2,b_3,b_4)$. Computing the Lie derivative
$\mathcal{L}_vg$ and comparing the result to zero gives a system of PDEs for unknowns $v_1, v_2, v_3, v_4$; its solution takes form 
\begin{align*}
{ v_1} &=\ \ { c_1}{ b_2}+{ c_2}{ b_3}+{
 c_3}{ b_4}+{ c_4},\ \ { v_2}  =-{ c_1}{ b_1}
+{ c_5}{ b_3}+{ c_6}{ b_4}+{ c_7},\\
{ v_3} &=-{ c_2}{ b_1}-{ c_5}{ b_2}+{ c_8}{ 
b_4}+{ c_9},\ \ { v_4}  =-{ c_3}{ b_1}-{ c_6}{
 b_2}-{ c_8}{ b_3}+{ c_{10}}
\end{align*}
for constants $c_i$, $ i=1,\dots,10$.
Substituting into $v$, computing $\mathcal{L}_ve_i$ for $i=1,2,3,4$, contracting with defining forms $\alpha$ and comparing with zero gives another system of PDEs. Its solution takes form 
\begin{align*}
 { v_5} &=-{ c_1}b_1^{2}+ ( 
 ( { c_3}+{ c_5} ) { b_3}+ ( {
 c_6} -{ c_2} ) { b_4}+2{ c_7} ) { b_1}-{ c_1}
 b_2^{2}+ (  ( c_6-{ c_2} ) { b_3}
- ( { c_3}+{ c_5} ) { b_4}\\ &-2{ c_4}
 ) { b_2}+{ c_8}{{ b_3}}^{2}+{ c_8} b_4^{2
}+{ a_3}{ c_5}-{ a_3}{ c_3}+{ a_4}{ c_6}+2{
 c_9}{ b_4}-2{ c_10}{ b_3}+{ a_4}{ c_2}+{
 c_{11}},\\ 
{ v_6}  &=-{ c_2} b_1^{2}+ ( 
 -( { c_3}+{ c_5} ) { b_2}+ ( { c_1}+{
 c_8} ) { b_4}+2{ c_9} ) { b_1}-{ c_6}
b_2^{2}+ (  -( { c_1}+{ c_8} ) { b_3}
+2{ c_{10}} ) { b_2}\\ 
&-{ c_2} b_3^{2}- ( 
 ( { c_3}+{ c_5} ) { b_4}+2{ c_4} ) 
{ b_3}-{ c_6} b_4^{2}-{ a_2}{ c_5}+{ a_2}{
 c_3}-2{ c_7}{ b_4}+{ a_4}{ c_8}-{ a_4}{ 
c_1}+{ c_{12}},\\
{ v_7}  &=-{ c_3} b_1^{2}+
 (  ( { c_2}-{ c_6} ) { b_2}- ( { 
c_1}+{ c_8} ) { b_3}+2{ c_{10}} ) { b_1}+{
 c_5} b_2^{2}+ (  -( { c_1}+{ c_8}
 ) { b_4}-2{ c_9} ) { b_2}\\
&+{ c_5}{{ b_3}
}^{2}+ (  ({ c_6} -{ c_2} ) { b_4}+2{ 
c_7} ) { b_3}-{ c_3}b_4^{2}-{ a_2}{ c_6}
-{ a_2}{ c_2}-2{ c_4}{ b_4}-{ a_3}{ c_8}+{
 a_3}{ c_1}+{ c_{13}}  
\end{align*}
Suitable independent choices of constants $c_i$, $i=1,\dots,13$, then give all generators of the symmetry algebra.

One can check by direct computation that the first seven fields $t_i$, $i=1,\dots,7$, commute with left-invariant vector fields $e_i$, $i=1,\dots,7$, and have the same Lie bracket structure as $ e_i$ up to the sign (compare to \eqref{zavorky}), so they are right-invariant and generate translations.
Finally, one can check that the vector fields $s_1, s_2, s_3$ and $s_4, s_5, s_6$ generate two independent Lie algebras $\frak{sp}(1)$, i.e.
\begin{align*}
&[s_1, s_2] = 2s_3, \quad [s_1, s_3] = -2s_2, \quad [s_2, s_3] = 2s_1, \\
&[s_4, s_5] = 2s_6, \quad [s_4, s_6] = -2s_5, \quad [s_5, s_6] = 2s_4,
\end{align*}
so the stabilizer is isomorphic to a direct product of two Lie algebras $\frak{sp}(1)$.
\end{proof}

Our choice of generators is such that  translations $t_1, t_2, t_3$ obviously correspond to the $a$--block, and
 the translations $t_4, t_5, t_6, t_7$ correspond to the $b$--block of the matrix \eqref{mat-sym}.
 Moreover, symmetries $s_1,s_2,s_3$ generate $\frak{sp}(1)$ corresponding to the $c$--block, while the symmetries $s_4, s_5, s_6$ generate $\frak{sp}(1)$ corresponding to the $d$--block of the matrix \eqref{mat-sym}. 
Indeed, the action of the element 
$$
\left[
\begin{smallmatrix}
C & 0 & 0 \\
0 & D & 0 \\
0 & 0 & C 
\end{smallmatrix}
\right] 
$$ 
for unit quaternions $C,D$ corresponding to $\Im{(c)}, \Im{(d)}$ results in
\begin{align*}
 \log ( \left[
\begin{smallmatrix}
C & 0 & 0 \\
0 & D & 0 \\
0 & 0 & C 
\end{smallmatrix}
\right]
\exp ( \left[
\begin{smallmatrix}
0 & 0 & 0 \\
b & 0 & 0 \\
\Im(a) & -\bar b & 0 
\end{smallmatrix}
\right] )
\left[
\begin{smallmatrix}
 \bar C & 0 & 0 \\
0 & \bar D & 0 \\
0 & 0 & \bar C 
\end{smallmatrix}
\right]
)
=
\left[
\begin{smallmatrix}
0 & 0 & 0 \\
D b \bar C & 0 & 0 \\
C\Im(a)\bar C & -C\bar b \bar D & 0 
\end{smallmatrix}
\right].
\end{align*}

In the next, we denote the corresponding $\frak{sp}(1)$ by the subscript $c$ or $d$. Vector fields $s_4, s_5, s_6$ generating the algebra $\frak{sp}(1)_d$ correspond to the so called left-isoclinic rotations on the distribution $H$ in the standard basis planes (there are two orthogonal planes rotating in the same orientation by the same angle considering the basis $\partial_{b_{i}}, i=1,2,3,4$) and identity on $\partial_{a_{i}}, i=2,3,4$. 
Vector fields $s_1,s_2,s_3$ generating $\frak{sp}(1)_c$ correspond to the so called right-isoclinic rotations on the distribution $H$ in the standard basis planes (there are two orthogonal planes rotating in the opposite orientation by the same angle considering the basis $\partial_{b_{i}}, i=1,2,3,4$)
 composed with the rotation around a standard basis vector in the double angle in $3$-dimensional space given by $\partial_{a_i}, i =2,3,4$.
 
This clearly reflects the fact that the left multiplication of a quaternion by a unit quaternion represents a left-isoclinic rotation and the right multiplication represents a right-isoclinic rotation in $4$-dimensions, thus their compositions give arbitrary double rotations, i.e. elements of $SO(4,\R)$. Then the conjugation by a unit quaternion preserves a plane and gives a rotation in its orthogonal complement in the double angle.  
In particular, restricting to pure imaginary quaternions, it descends to a rotation in $3$-dimensional space.

\section{Local control and geodesics}

\subsection{Control problem and controls}
\label{sec2.1}
Let us now discuss the control system corresponding to the sub-Riemannian structure $(\mathcal{H},H,g)$. Considering (real) coordinates $(a_2,a_3,a_4,b_1,$  $b_2,b_3,b_4)$, the optimal control problem takes form 
\begin{align}
\label{control1}
\dot c(t)=
 u_1\begin{bmatrix}b_2 \\ b_3 \\ b_4 \\ 1 \\0\\0\\0
\end{bmatrix} +
u_2\begin{bmatrix}
-b_1 \\ \ b_4 \\ \ b_3 \\ \ 0 \\ \ 1\\ \ 0\\ \ 0
\end{bmatrix}  +
u_3\begin{bmatrix}
\ b_4 \\ \ b_1 \\ -b_3 \\ \ 0 \\ \ 0\\ \ 1\\ \ 0
\end{bmatrix}  +
u_4\begin{bmatrix}
-b_3 \\ \ b_2 \\ \ b_1 \\ \ 0 \\ \ 0\\ \ 0\\ \ 1
\end{bmatrix}  
\end{align}
for $t>0$ and $c \in \mathcal{H}$ and the control $u=(u_1,u_2,u_3,u_4) \in \R^4$ with the boundary condition $c(0)=c_1$ and $c(T)=c_2$ for fixed points $c_1,c_2 \in \mathcal{H}$, where we minimize the energy
\begin{align}
\label{control2}
\frac12\int_0^T (u_1^2+u_2^2+u_3^2+u_4^2)dt.
\end{align}

We follow here Hamiltonian concepts to approach the control system and we use its left-invariancy, \cite[Sections 7 and 13]{ABB}. The left-invariant vector fields $e_i$, $i=1,2,3,4$, $2e_i$, $i=5,6,7$ form a basis of $T\mathcal{H}$ and determine left-invariant coordinates on $\mathcal{H}$. We consider the factor $2$ here to avoid it in the subsequent equations. The corresponding left-invariant coordinates $h_i$, $i=1,\dots,7$ on fibers of $T^*\mathcal{H}$ are given by the evaluation $h_i (\lambda )=  \lambda(e_i)$, $i=1,2,3,4$ and $h_i (\lambda )=2\lambda(e_i)$, $i=5,6,7$, respectively, for arbitrary one-form $\lambda$ on $\mathcal{H}$. Thus we can use $(a_j,b_k,h_\ell)$, $j=2,3,4$, $k=1,2,3,4$, $\ell=1,\dots,7$, as global coordinates on $T^*\mathcal{H}$. 
In these coordinates, local minimizers, aka geodesics, are projections of the integral curves for Pontryagin's maximum principle system corresponding to the sub-Riemannian Hamiltonian 
\begin{align}
\label{hamiltonian}
K:=\frac12 (h_1^2+h_2^2+h_3^2+h_4^2)
\end{align}
 to the base space $\mathcal{H}$.
\begin{rem*}
For $2$-step filtrations, the projection of each abnormal geodesic coincides with the projection of a normal geodesic, i.e. there are no strict abnormal geodesics, \cite{Jean,Goh}. Thus it is sufficient to study the normal Hamiltonian \eqref{hamiltonian}.
\end{rem*}

Due to the fact that we have a left-invariant Hamiltonian on a Lie group, the fiber part of the system is given by the co-adjoint action of the differential of $dK(o)=\sum_{i=1}^4 h_ie_i$ at the origin $o \in \mathcal{H}$. Viewed as a linear endomorphism in our left-invariant basis, the adjoint action takes form 
\begin{align}
\label{adjoint-action}
\left[ \begin {smallmatrix} 0&0&0&0&0&0&0\\ 0&0&0
&0&0&0&0\\ 0&0&0&0&0&0&0\\ 0&0&0&0
&0&0&0\\ -{ h_2} ( t ) &{ h_1}
 ( t ) &{ h_4} ( t ) &-{ h_3} ( t
 ) &0&0&0\\ -{ h_3} ( t ) &-{ 
h_4} ( t ) &{ h_1} ( t ) &{ h_2} ( t
 ) &0&0&0\\-{ h_4} ( t ) &{ 
h_3} ( t ) &-{ h_2} ( t ) &{ h_1} ( t
 ) &0&0&0\end {smallmatrix} \right],
\end{align}
so the co-adjoint action is given by its transpose, i.e. we mupltiply a row vector $(h_1, \dots, h_7)$ by the matrix from the right.

For $i=5,6,7$ we get $\dot h_i=0$ and thus 
$h_i$ constant. For suitable constants then we denote 
\begin{align} \label{h567}
h_5=C_5,\ \ h_6=C_6,\ \ h_7=C_7.
\end{align} 
 Then for $h=(h_1,h_2,h_3,h_4)^t$ we get a linear system with constant coefficients
$
\dot h =  \Omega h
$
for the matrix
\begin{align}
\label{Omega}
\Omega=
\left[ 
\begin{smallmatrix} 
0&-{ C_5}&-{ C_6}&-{ C_7}
\\  {C_5}&0&{ C_7}&-{ C_6}
\\ { C_6}&-{ C_7}&0&{ C_5}
\\ { C_7}&{ C_6}&-{ C_5}&0
\end{smallmatrix}
 \right].
\end{align}

\begin{rem*}
The matrix $\Omega$ reflects minus the multiplication table of the Lie algebra $\fh$; that is typical behaviour for $2$-step filtrations, \cite{ABB}.
\end{rem*}

The solution of the system $\dot h =  \Omega h$ is given by  
$h(t)= e^{t \Omega} h(0)$, where  $h(0)$ is the initial value of the vector $h$ in the origin. 
If $C_5=C_6=C_7=0$, then $h(t)=h(0)$ is constant. To avoid these degenerate solutions, we assume that the vector $(C_5,C_6,C_7)$ is non--zero and 
we denote by $C$ its length $$C=\sqrt{C_5^2+C_6^2+C_7^2}.$$ 

\begin{prop*} \label{vert-p} The general solutions of the system $\dot h=\Omega h$ satisfying \eqref{h567} for  $C \neq 0$ take form   
\begin{align} 
\label{hacka}
 h(t)=\cos(Ct)u+\sin(Ct)v
\end{align}
for orthogonal vectors 
\begin{align} \label{uv}
u=\left[ 
\begin {smallmatrix} {\frac {{ C_6}{ C_7}{ C_1}+{ C_5
}C{ C_2}-{ 
C_5}{ C_7}{ C_3}+{ C_6}C{ C_4}}{ C_5^{2}+C_6^{2}}}
\\ { C_1}\\ { C_3}
\\ {\frac {{ C_5}{ C_7}{ C_1}-{ C_6}
C{ C_2}+{ C_6}
{ C_7}{ C_3}+{ C_5}C{ C_4}}{ C_5^{2}+ C_6^{2}}}\end {smallmatrix}
 \right],\ 
v=
\left[ 
\begin {smallmatrix} -{\frac {C{ C_5}{ C_1}-{ C_6}{ C_7}{ C_2}+
C{ C_6}{ C_3}
+{ C_4}{ C_5}{ C_7}}{ C_5^{2}+ C_6^{2}}}
\\ { C_2}\\ { C_4}
\\ {\frac {C{ C_6}{ C_1}+{ C_2}{ C_5}{ C_7}-C{ C_5}{ C_3}+{ C_4}
{ C_6}{ C_7}}{ C_5^{2}+ C_6^{2}}}\end {smallmatrix}
 \right] 
\end{align}
of the same length $\sqrt{D}$ for
\begin{align} \label{D}
D={\frac {2{ C_7} \left( { C_1}{ C_4}-{ C_2}{ C_3}
 \right) C+ \left(  C_1^{2}+ C_2^{2}+ C_3^{2}+ C_4^{2} \right)C^2 }{ C_5
^{2}+ C_6^{2}}},
\end{align}
in the case $(C_5,C_6)\neq 0$, and
for orthogonal vectors 
\begin{align} \label{uv7}
u=\left[ 
\begin {smallmatrix} 
C_4 \\ -C_2 \\ C_1 \\ C_3
\end{smallmatrix}
 \right],\ 
v= 
\left[ 
\begin {smallmatrix} 
-C_3 \\ C_1 \\ C_2 \\ C_4
\end{smallmatrix}
 \right] 
\end{align}
of the same length $\sqrt{D}$ for
\begin{align} \label{D7}
D= C_1^{2}+ C_2^{2}+ C_3^{2}+ C_4^{2},
\end{align}
in the case $(C_5,C_6)= 0$, where $C_1,C_2,C_3,C_4$ are real constants. 
\end{prop*}

\begin{proof}
The solution of the system is given by exponential of the matrix $\Omega$, so we need to analyze its eigenvalues and eigenvectors.  
It follows that there are (complex conjugated) imaginary eigenvalues $\pm iC$  both of algeraic as well as geometric multiplicity two. In the case $(C_5,C_6)\neq 0$, the corresponding eigenspace of $iC$ is generated by two independent complex eigenvectors  
$$
\left[ 
\begin {smallmatrix} {\frac {iC{ C_5}+{ C_6}{ C_7}}{ C_5^{2}+ C_6^{2}}}\\ 1\\ 0
\\ {\frac {-iC{ C_6}+{ C_5}{ C_7}}{ C_5^{2}+ C_6^{2
}}}\end {smallmatrix} 
\right], \ \ 
\left[ 
\begin {smallmatrix} -{\frac {-iC{ C_6}+{ C_5}{ C_7}}{ C_5^{2}+ 
C_6^{2}}}\\ 0\\ 1
\\ {\frac {iC{ C_5}+{ C_6}{ C_7}}{ C_5^{2}+ C_6^{2
}}}\end {smallmatrix} 
\right] 
$$
and the eigenspace for $-iC$ is complex conjugate.
In the case $(C_5,C_6)=0$ and $C_7$ the only non-zero contant, the  corresponding eigenspace of $iC_7$ is generated by two independent complex eigenvectors
$$
\left[ 
\begin {smallmatrix} 
0 \\ -i \\ 1 \\0
\end {smallmatrix} 
\right], \ \ 
\left[ 
\begin {smallmatrix} 
i \\0\\0\\1
\end {smallmatrix} 
\right] 
$$
and the eigenspace for $-iC_7$ is complex conjugate.
Writting the complex solutions
\[
e^{iCt}
\left[ 
\begin {smallmatrix} {\frac {iC{ C_5}+{ C_6}{ C_7}}{ C_5^{2}+ C_6^{2}}}\\ 1\\ 0
\\ {\frac {-iC{ C_6}+{ C_5}{ C_7}}{ C_5^{2}+ C_6^{2
}}}\end {smallmatrix} 
\right], 
\quad 
e^{iCt}
\left[ 
\begin {smallmatrix} -{\frac {-iC{ C_6}+{ C_5}{ C_7}}{ C_5^{2}+ 
C_6^{2}}}\\ 0\\ 1
\\ {\frac {iC{ C_5}+{ C_6}{ C_7}}{ C_5^{2}+ C_6^{2
}}}\end {smallmatrix} 
\right]\,, 
\quad\text{ and }\quad
e^{iC_7t}
\left[ 
\begin {smallmatrix} 
0 \\ -i \\ 1 \\0
\end {smallmatrix} 
\right], 
\quad
e^{iC_7t}
\left[ 
\begin {smallmatrix} 
i \\0\\0\\1
\end {smallmatrix} 
\right]\,,
\]
respectively,
utilizing the Gauss formula,
decomposing both couples of these complex solutions to real and imaginary parts,
considering the combinations for constants $C_1,C_2,C_3,C_4$ and reordering gives \eqref{hacka}.
It is a direct computation to check the orthogonality of vectors $u, v$ and to find their lengths.
\end{proof} 

Let us discuss some properties of the solutions. In particular, the following observation holds in general. Denote $(\cdot , \cdot)$ the scalar product and $\| \cdot \|$ the corresponding norm.

\begin{prop*} \label{circle-p} 
Let $\beta\in\R$ be arbitrary.
Let us assume there are two vectors $u,v\in\R^n$ of the same length 
such that $(u,v)=0$, i.e., they are othogonal.
Then $$\|\cos(\beta t)u + \sin(\beta t) v\|=\|u\|=\|v\|$$ for any $t\in\R$.
In particular, the curve given by the parametric formula
$$X:\R\to\R^n, \ \ 
X(t)=\cos(\beta t)u + \sin(\beta t) v, \ \ t\in\R,$$ is a circle 
in the plane $\langle u,v \rangle$ given by the vectors $u$ and $v$
and centered at the origin.
\end{prop*}

\begin{proof}
Direct computation with the help of known properties of goniometric functions gives
\[
\begin{array}{rcl}
\|\cos(\beta t)u + \sin(\beta t) v\|^2
&=&
\left(\cos(\beta t)u + \sin(\beta t) v, \cos(\beta t)u + \sin(\beta t) v\right)
\\
&=&
\cos^2(\beta t)\|u\|^2 + \sin^2(\beta t) \|v\|^2 +2(u,v)\cos(\beta t)\sin(\beta t) 
\\
&=&
\|u\|^2=\|v\|^2
\end{array}
\]
for any $\beta,t\in\R$ and the rest follows.
\end{proof}

\begin{cor*}\label{circle-c}
The function $h$ from~\eqref{hacka} describes a circle
centered at the origin 
and lying in the plane $\langle u,v \rangle$,
where $u$ and $v$ being
from~\eqref{uv} or~\eqref{uv7}, 
respectively,
\end{cor*}

In the next, we assume that the vector of constants $(C_1,C_2,C_3,C_4)\neq 0$ to avoud degenerate solution.

\subsection{Geodesics} \label{sec2.2}
The base part of the system is generally given as $\dot c(t)=\sum_{i=1}^4h_ie_i$ for a geodesic $c(t)$, control functions $h_i$ and generators of the control distribution $e_i$. Direct computation gives that our base system for $b=(b_1,b_2,b_3,b_4)$ and $a=(a_2,a_3,a_4)$ takes form 
\begin{align}
\begin{split}
\label{base-system}
\dot b& =h, \\ 
\dot a&=Ab,
\end{split}
\end{align}
where $A$ is the non-trivial submatrix of \eqref{adjoint-action}, i.e.
$$A=\left[ \begin {smallmatrix}  -{ h_2} ( t ) &{ h_1}
 ( t ) &{ h_4} ( t ) &-{ h_3} ( t
 ) \\ -{ h_3} ( t ) &-{ 
h_4} ( t ) &{ h_1} ( t ) &{ h_2} ( t
 ) \\-{ h_4} ( t ) &{ 
h_3} ( t ) &-{ h_2} ( t ) &{ h_1} ( t
 ) \end {smallmatrix} \right].$$

\begin{prop*}
Solutions of the system \eqref{base-system} starting at the origin $o=(0,\dots,0)$ take form
\begin{align}
\begin{split}
\label{ab}
b(t)&=\frac{1}{C}\sin(Ct)u+\frac{1}{C}(1-\cos(Ct))v,\\
a(t)&=\frac{1}{C^2}(Ct-\sin(Ct))w,
\end{split}
\end{align}
where $u,v$ are vectors \eqref{uv}, and denoting by $w_{ij}=(u \wedge v)_{ij}$ the components of the wedge product $u\wedge v$
\begin{align} \label{w-as-wedge}
w=\left[ \begin {smallmatrix} 
w_{21}+w_{34} \\
w_{31}+w_{42} \\
w_{23}+w_{41}
\end {smallmatrix} \right]. 
\end{align} 
\end{prop*}
\begin{proof}
We find the solution $b(t)$ by a direct integration of $h(t)$ and involving the initial condition. Then computing $Ab(t)$ and subsequent direct integration gives $a(t)$ as proposed. The description of $w$ the follows. 
\end{proof}

Explicitly, denoting $u=[u_i], v=[v_i]$ components of the vectors $u, v$, the vector $w$ takes form
$$
w= \left[ \begin {smallmatrix} -u_{{1}}v_{{2}}+u_{{2}}v_{{1}}+u_{{3
}}v_{{4}}-u_{{4}}v_{{3}}\\ -u_{{1}}v_{{3
}}-u_{{2}}v_{{4}}+u_{{3}}v_{{1}}+u_{{4}}v_{{2}}
\\ -u_{{1}}v_{{4}}+u_{{2}}v_{{3}}-u_{{3}}v
_{{2}}+u_{{4}}v_{{1}}\end {smallmatrix} \right].
$$
Writing explicitly $w$ using the constants $C_i$ and $C,D$, we end up with 
$$
w= -\frac{D}{C} \left[ \begin{smallmatrix}
C_5 \\ C_6 \\C_7
\end{smallmatrix}
\right].
$$

Let emphasize that the parametrization of geodesics is encoded in level sets of the Hamiltonian \eqref{hamiltonian}
of the system. In particular, the arc length parametrization corresponds to the level set $K=\frac12$, \cite{ABB}. In our case, the level sets correspond to the choice that $D$ is constant (see \eqref{D} for $D$) and arc length parametrization corresponds to $D=1$.

Finally, it turns out that for each geodesic $c(t)=(a(t),b(t))$,  the component $b(t)$ corresponds to a circle starting at the origin, while $a(t)$ corresponds to a line. We discuss some examples in Section \ref{sec3.3} in detail.

\section{Symmetries and geodesics}

\subsection{Action symmetries on geodesics} \label{sec3.1}
Let us consider symmetries corresponding to the algebra $\frak{sp}(1)_d$ that determine double rotations in the distribution $H$ around the origin in the same angle and identity on the $a$-part. 
Each non-degenerate geodesic $c(t)=(b(t),a(t))$ satisfies that $b(t)$ is a linear combination of two orthogonal vectors $u$ and $v$ of the same length; thus the curve $b(t)$ lies in the plane $\langle u,v \rangle$. 
 
Let us remind that a Maxwell point is a point where two different geodesics meet each other with the same values of the cost functional and time (called a Maxwell time).
We use the above symmetries to discuss Maxwell points of the geodesics.

\begin{prop*} \label{sym-d}
For each non-degenerate geodesic $c(t)=(a(t),b(t))$ starting at the origin, there is a family of geodesics of the same parametrization starting at origin all intersecting at the same Maxwell time $T=\frac{2\pi}{C}$ at the Maxwell point
$$c(T)=\left[-\frac{2\pi D}{C^3}C_5,-\frac{2\pi D}{C^3}C_6,-\frac{2\pi D}{C^3}C_7,0,0,0,0\right].$$ 
In particular, the geodesic $c(t)$ is no more optimal after it reaches the time $T$.
\end{prop*}
\begin{rem*}
Let us remind that the (family of) geodesics carries the arc-length parametrization for $D=1$.
\end{rem*}
\begin{proof}
The curve $b(t)$ lies in the plane $\langle u,v \rangle \subset H$ through origin $o$. Each rotation $S \in \exp(\frak{sp}(1)_d) \simeq SO(3,\R)_d$ of $H$ rotates the plane around $o$ or maps the plane to a different plane in $H$ through $o$. In this way, we get a family of generally different geodesics through $o$ sharing the same parametrization as $b(t)$, and the family is the orbit of $b(t)$ with respect to the action of $SO(3,\R)_d$, i.e.
$$S(b(t))=\frac{1}{C}\sin(Ct)S(u)+\frac{1}{C}(1-\cos(Ct))S(v).
$$
Thus we need to focus on the coefficients in the combination. It holds for $t=\frac{2\pi k}{C}$
\begin{gather*}
 \sin(C\cdot \frac{2\pi k}{C})=0, \ \ \ 
1-\cos(C\cdot \frac{2\pi k}{C}) = 0
\end{gather*}
and first such positive time appears for $k=1$. Thus all geodesics from the orbit of $b(t)$ intersect for the time $T=\frac{2\pi}{C}$.

Since the symmetries corresponding to $\frak{sp}(1)_d$ preserve $w$, they
 only impact the curve $b(t)$ and we get identity on $a(t)$.
So we get the proposed family of curves as the orbit of the action of the symmetries corresponding to $\frak{sp}(1)_d$ acting on $c(t)$; the rest is a direct computation.
\end{proof}

\begin{rem*}
This mimics the behaviour known from the local control on the classical Heisenberg group, \cite{mope,hzn}.
\end{rem*}

Let us finally comment on symmetries corresponding to the algebra $\frak{sp}(1)_c$. They also determine double rotations in the distribution $H$ around the origin in the same angle. Thus if they were to induce some other restrictions on solutions, they would in particular give the same restrictions to the part $b(t)$ of the solution. 
 The same principle then applies to each symmetry corresponding to $\frak{so}(4,\R)$, thus we cannot get better estimate in this way.

\subsection{Moduli space}
\label{sec3.2}
For each geodesic $c(t)$ and thus for the corresponding vectors $u$, $v$ and $w$, there is a symmetry $R$ such that 
\begin{align} \label{rot}
R:\,
u\mapsto u_1:= \left[ \begin {smallmatrix} 
\sqrt{D} \\0\\0\\0
\end{smallmatrix}
\right], \ \ v\mapsto v_1 := \left[ \begin {smallmatrix} 
0\\ \sqrt{D}\\0\\0
\end{smallmatrix}
\right], \ \ w \mapsto w,
\end{align}
so we can view $b(t)$ as a curve lying in the basis plane of the first two standard basis vectors.
Thus each geodesic can be represented by the geodesic of the form
\begin{align}
\begin{split} \label{factorgeodesic}
b_1(t)&=\frac{D}{C}\sin(Ct)\\
b_2(t)&=\frac{D}{C}(1-\cos(Ct))\\
b_3(t)&=0\\
b_4(t)&=0\\
a_2(t)&= -\frac{DC_5}{C^3}(Ct-\sin(Ct))\\ 
a_3(t)&= -\frac{DC_6}{C^3}(Ct-\sin(Ct))\\ 
a_4(t)&= -\frac{DC_7}{C^3}(Ct-\sin(Ct)).
\end{split}
\end{align}
This means that we can factorize the solution space by the action of the symmetries of $SO(3,\R)_d$ to 
$$\mathcal{H}_d:=\mathcal{H}/SO(3,\R)_d$$
and each geodesic \eqref{factorgeodesic} is the representative of a class in the quotient space.

Denote $(\cdot , \cdot)$  the scalar product on $\R^4$. The (square of the) length $(b,b)$ of the vector $b=(b_1,b_2,b_3,b_4)$ is an invariant of the action of $SO(3,\R)_d$.  This means that $b_d :=(b,b)$ together with $a_2, a_3, a_4$ determine (local) coordinates on $\mathcal{H}_d \simeq \R^4=\R \oplus \R^3$. Geodesics then (locally) descend to curves of the form
\begin{align}
\begin{split} \label{faktorkrivky}
b_d(t)&=\frac {2 D^2}{C^2}(1-\cos \left( Ct \right)) \\
a_2(t)&= -\frac{DC_5}{C^3}(Ct-\sin(Ct))\\ 
a_3(t)&= -\frac{DC_6}{C^3}(Ct-\sin(Ct))\\ 
a_4(t)&= -\frac{DC_7}{C^3}(Ct-\sin(Ct)).
\end{split}
\end{align}

Following \cite{mya1,ampa}, we can use these observations to discuss (locally, around the origin) critical points of the exponential map 
\begin{align*}
& \exp:\left(\left\{K=\frac12\right\} \cap T_o^*\R^7\right) \times \R \to \R^7, \\ 
&\exp(u,C_5,C_6,C_7,t) =(a(t),b(t))
\end{align*} for the Hamiltonian \eqref{hamiltonian} with initial values $h(0)=u$, $h_i(0)=C_i$ for $i=5,6,7$,  $a(0)=0$, $b(0)=0$. Moreover, the level set $K=\frac12$ reflects the condition $D=1$ which means that the vector $u$ has the length equal to $1$.

Due to the invariancy with respect to $SO(3,\R)_d$, we can locally factorize the exponental map to the map
\begin{align}
\begin{split}
\label{factorexp}
&\exp_d:\left(\left\{K= \frac12 \right\} \cap T_o^*\R^4\right) \times \R \to \R^4, 
\\  
&\exp_d(C_5,C_6,C_7,t) =(a_2(t),a_3(t),a_4(t),b_d(t))
\end{split}
\end{align}
which is correctly defined if we assume $D=1$ that is equivalent to $K=\frac12$. Let us note that we can imagine this as working with geodesics \eqref{factorgeodesic} with $D=1$, which manifestly leads to curves \eqref{faktorkrivky} for $D=1$. Thus we focus on the Jacobian of the map 
$$
(C_5,C_6,C_7,\tau)\mapsto \left(\frac {2}{C^2}(1-\cos \left( \tau \right)),
-\frac{C_5}{C^3}(\tau-\sin(\tau)),
-\frac{C_6}{C^3}(\tau-\sin(\tau)), 
-\frac{C_7}{C^3}(\tau-\sin(\tau))\right),
$$
where we consider $\tau=Ct$, and its zero points. Direct computation gives that the Jacobian equals to
\begin{align}
\begin{split}
\label{jacobian}
J(\tau) =\,\,&4{\tau}^{3}\sin ( \tau ) +8{\tau}^{2}  \cos^{2}
 ( \tau )   -4\tau  \cos^{2} ( \tau
 )  \sin ( \tau ) + 
8{\tau}^{2}\cos
 ( \tau ) - 
 16\tau\cos ( \tau ) \sin ( 
\tau ) 
\\ &-16{\tau}^{2}-8 \cos^{3} ( \tau )+ 
 20\tau\sin ( \tau ) +8  \cos^{2}
 ( \tau )   +8\cos ( \tau ) -8
 \end{split}
 \end{align}
multiplied by $\frac{1}{C^{10}}$.
\begin{prop*}
The first positive zero of the function $J(\tau)$ from \eqref{jacobian} occures for $\tau=2\pi$.
\end{prop*}
\begin{proof}
The function $J(\tau)$ can be factorized to
\[
J(\tau)=
-4(\tau-\sin(\tau))^2 (\tau \sin(\tau)+2 \cos(\tau)-2)
\]
and it is easy to see that the zeros of $J$ are $\tau=0$ and the zeros of
\[
f(\tau):=
\tau \sin(\tau)+2 \cos(\tau)-2\,.
\]
Trivially, the points where $\sin(\tau)=0$ and simultaneously $\cos(\tau)=1$,
i.e. the points $\tau=2k\pi$, $k\in\Z$, are the zeros of $f$.
Observing, that $f$ is an even $C^\infty(\R)$-function and
\begin{align*}
f'(\tau) &= \tau \cos(\tau)-\sin(\tau)\,,
\\
f''(\tau) &= -\tau \sin(\tau)\,,
\end{align*}
we see that 
the function $f$ has positive extremal points $\tau=\tau_k^E$, each one being a zero
of $\tau=\tan(\tau)$ in the interval $(k\pi,(k+\frac12)\pi))$, $k\in\NN$.
Moreover, $\tau_k^E$ is the point of a local minimum or maximum
for $k$ odd or even, respectively.
Realizing that $f((2k+1)\pi)=-4<0$ for $k\in\Z$ we can conclude that
the remaining (besides those $2k\pi$-ones) \emph{positive} 
zeros of $f$ (and hence of $J$) lie strictly in the intervals
$\left(\tau_{2k}^E,(2k+1)\pi\right)$, $k\in\NN$.
In particular, the first (smallest) positive root of $J$ is the point 
$\tau=2\pi$.
\end{proof}
The function $J(\tau)$ for $\tau>0$ has the  shape as in Figure \ref{fceJ}.
\begin{figure}[h]
 \begin{center}
 \includegraphics[height=60mm]{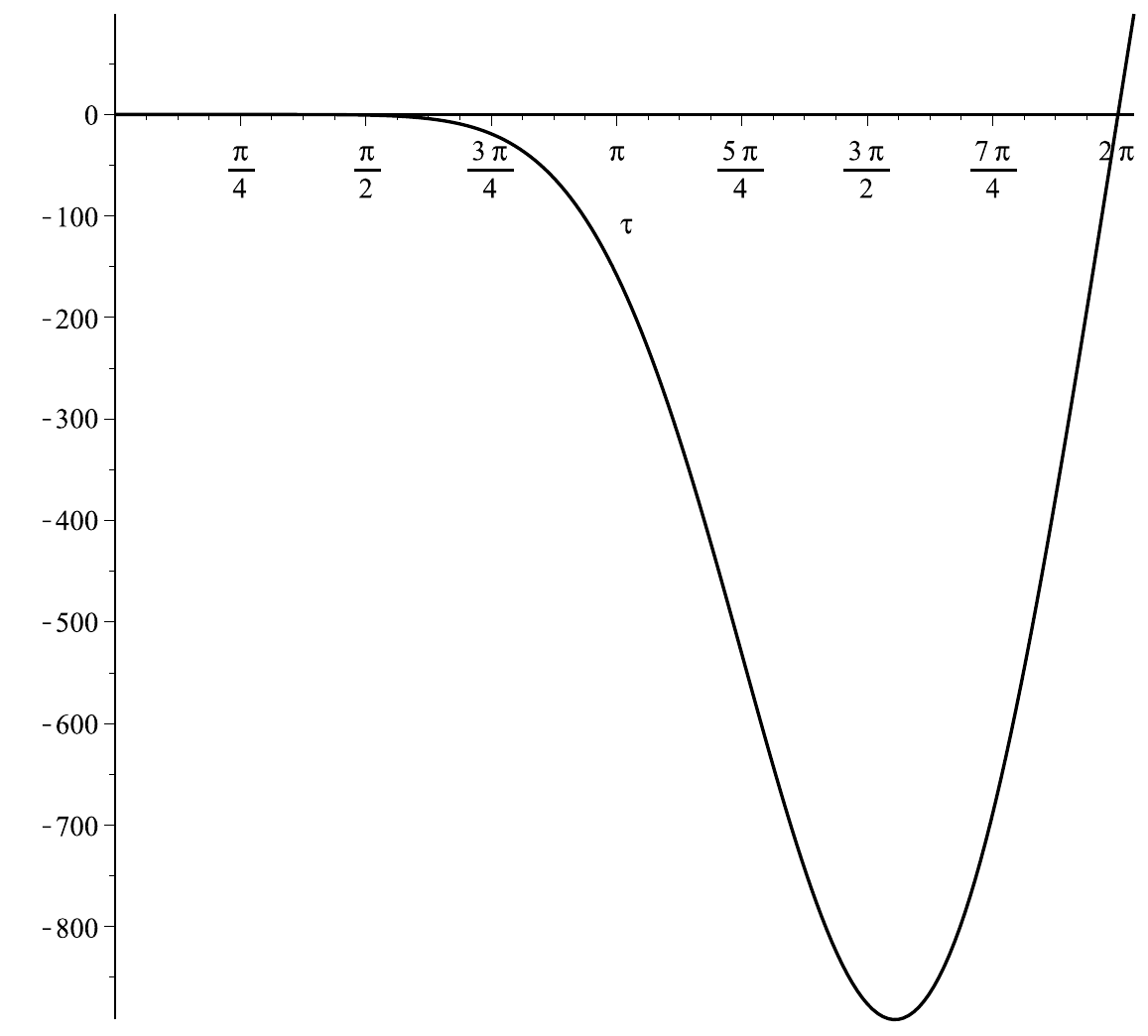}
 \caption{The graph of the function $J(t)$}
 \label{fceJ}
 \end{center}
 \end{figure}

Thus for the curves \eqref{faktorkrivky}, the time $T=\frac{2\pi}{C}$ is the first conjugate time. Going to the preimage of the quotient space, the same holds for geodesics \eqref{ab} of the original control problem.
Then from \cite[Theorem 8.72.]{ABB} we get the following statement.
\begin{cor*}
For geodesics \eqref{ab}, the time $T=\frac{2\pi}{C}$ is the cut time.
\end{cor*}

\subsection{Examples of orbits}
\label{sec3.3}
Let us show here some examples of the actions on a specific geodesic. Consider the arc-length parametrized geodesic $c(t)=(a(t),b(t))$ such that 
$$
a(t)=[\sin(t)-t, 0, 0]^t,\ \ b(t)=[\sin(t), 1-\cos(t), 0, 0]^t.
$$
Thus $b(t)$ is the combination of $u=[1,0,0,0]^t$ and $v=[0,1,0,0]^t$ and according to \eqref{w-as-wedge}, the vector $w$ takes form $[-1,0,0]^t$.
Let us note that the geodesic corresponds to the choice of constants $C_2=C_5=1$ and remaining constants $C_i$ vanish.

The orbit of the geodesic $c(t)$ by means of $a(t), b(t)$ with respect to the action of $SO(3,\R)_c$ is as follows
 \begin{align*} 
 S(a(t))&=\frac{1}{q^2}\left[ \begin {smallmatrix} - ( t-\sin ( t )  ) 
 (  (  q_3^{2}+ q_2^{2} ) \cos ( 2q
s ) + q_1^{2} ) \\  ( t-\sin
 ( t )  )  ( { q_1}{ q_2}\cos ( 2
qs ) +{ q_3}q\sin ( 2qs ) -{ q_1}{ q_2}
 ) \\ - ( t-\sin ( t ) 
 )  ( -{ q_1}{ q_3}\cos ( 2qs ) +{ 
q_2}q\sin ( 2qs ) +{ q_1}{ q_3} ) 
\end {smallmatrix} \right]
 \\
 S(b(t))&=\frac{1}{q}\left[ \begin {smallmatrix} { q_1} ( -1+\cos ( t ) 
 ) \sin ( qs ) +\sin ( t ) \cos ( qs
 ) q\\ -q ( -1+\cos ( t ) 
 ) \cos ( qs ) +\sin ( qs ) \sin ( t
 ) { q_1}\\ \sin ( qs )  ( 
\sin ( t ) { q_2}+{ q_3}\cos ( t ) -{ 
q_3} ) \\ \sin ( qs )  ( \sin
 ( t ) { q_3}-{ q_2}\cos ( t ) +{ q_2}
 ) \end {smallmatrix} \right]
 \end{align*}
where we write the action of the general element $S=q_1s_1+q_2s_2+q_3s_3$ (see Proposition~\ref{prop-symmetries} for notation) and $$q=\sqrt{q_1^2+q_2^2+q_3^2}.$$

\begin{rem*}
One can see the action corresponding to $\frak{sp}(1)_c$ as the quaternionic aka rotational grading element which reflects the behaviour on geodesics.
\end{rem*}

We see that $t=2\pi$ is the first conjugate time for $b(t)$ and we shall choose $q_2=q_3=0$ to get a family of curves $c(t)$ intersecting at one point which is $P=[-2\pi,0,0,0,0,0,0]^t$. It is the orbit for the action corresponding to $s_1$ which gives a family of rotations in the plane $\langle u, v \rangle$ acting on $b(t)$. We can vizualise the circle $b(t)$ (red) and its orbit in the plane as in Figure \ref{orbita}. 
\begin{figure}[h]
 \begin{center}
 \includegraphics[height=60mm]{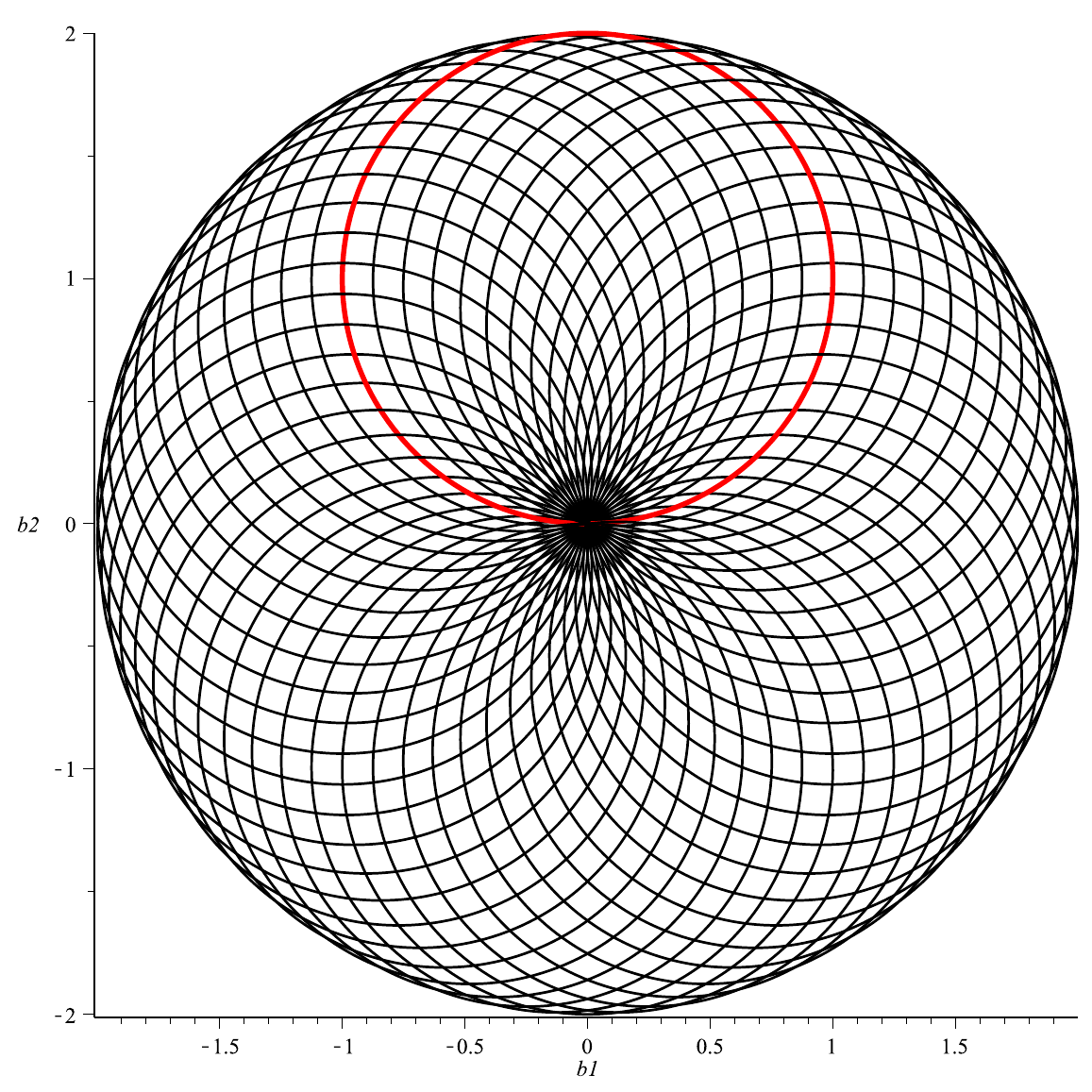}
 \caption{The red curve $b(t)$ and its orbit in the plane}
 \label{orbita}
 \end{center}
 \end{figure}
 
 \noindent
 In particular, it manifestly reminds the geodesics on real Heisenberg group.
 
The action of the one-parametric family (parametrized by $s$) given by $s_1$ preserves $w$ which is the common axis of the corresponding rotations in double angle; that is why it also works for $a(t)$. However, the general action does not preserve the component $a(t)$ of the solution. The orbit of $a(t)$ for the symmetry corresponding to the choice $q_1=q_2=1$, $q_3=-1$ and parameter $s=0,\dots,1$ (from green to red) can be visualised as in Figure \ref{orbita2}.
\begin{figure}[h]
 \begin{center}
 \includegraphics[height=60mm]{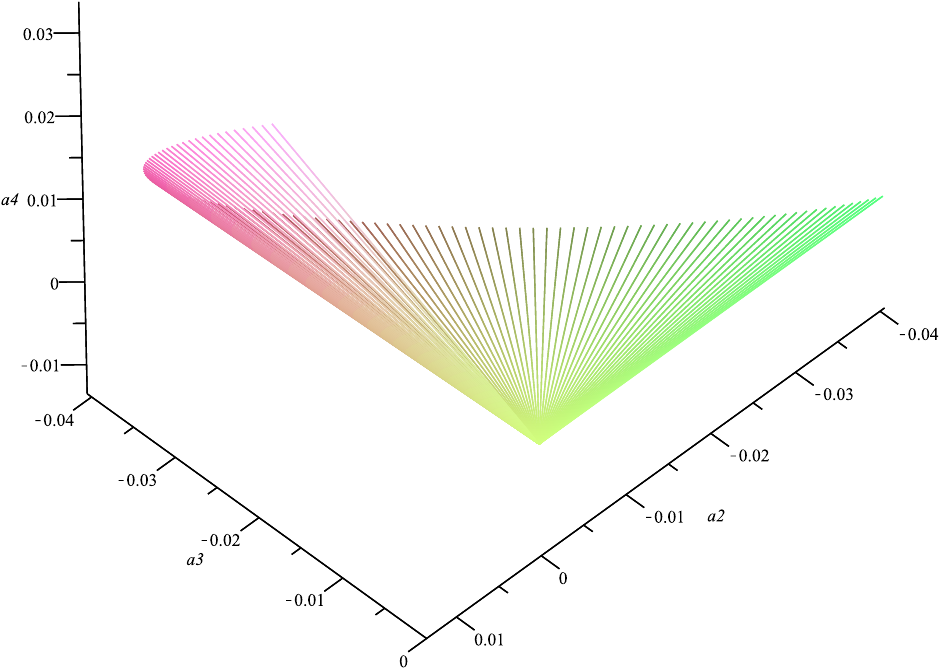}
 \caption{The orbit of $a(t)$}
 \label{orbita2}
 \end{center}
 \end{figure}
 
Moreover, there is a one more curve with the same property for the choice of the parameter $s=\pi$ of the form
$$
a(t)=[\sin(t)-t, 0, 0]^t, \ \ b(t)=[ -\sin(t), \cos(t)-1, 0, 0]^t
$$

Let us note that the action of the element $T=q_4s_4+q_5s_5+q_6s_6$
(see Proposition~\ref{prop-symmetries} for notation), 
i.e. of general element of $\SO(3,\R)_d$ on $b(t)$, is as follows
$$
T(b(t))=\frac{1}{q}\left[ \begin {smallmatrix} -{ q_4} ( \cos ( t ) -1
 ) \sin ( qs ) +\sin ( t ) \cos ( qs
 ) q\\ -q ( \cos ( t ) -1
 ) \cos ( qs ) -\sin ( qs ) \sin ( t
 ) { q_4}\\ -\sin ( qs )  ( 
\sin ( t ) { q_5}-{ q_6}\cos ( t ) +{ 
q_6} ) \\ -\sin ( qs )  ( { 
q_5}\cos ( t ) +\sin ( t ) { q_6}-{ q_5}
 ) \end {smallmatrix} \right] 
$$
and $T(a(t))=a(t)$, so $T$ preserves $t=2\pi$ and we get three--parametric family of curves intersecting at one point $P$. Then it follows how it works for their composition, i.e. a general symmetry.

\end{document}